\documentclass[11pt]{article}
\usepackage{amscd}
\usepackage{graphicx}
\usepackage{caption}
\usepackage{amsthm}
\usepackage{amsfonts}
\usepackage{amsmath}
\usepackage{amssymb}
\usepackage{latexsym}
\usepackage{mathtools}
\usepackage{times}
\usepackage[T1]{fontenc}
\newcommand{\doublespacing}{\let\CS=\@currsize\renewcommand{\baselinestretch}{1.4}\tiny\CS}

\newtheorem{Theorem}{Theorem}[section]
\newtheorem{Lemma}[Theorem]{Lemma}
\newtheorem{Corollary}[Theorem]{Corollary}
\newtheorem{Proposition}[Theorem]{Proposition}
\newtheorem{Remark}[Theorem]{Remark}

\newtheorem{Example}[Theorem]{Example}

\newtheorem{Definition}[Theorem]{Definition}

\numberwithin{equation}{section}

\begin{document}
	
\title{Cohen-Macaulayness of Total Simplicial
Complexes}
\author{{\bf Najam Ul Abbas$^\ast$, Imran Ahmed$^\ast$ and Ayesha Kiran$^\ast$}\\
		{\small $^\ast$  Department of Mathematics, COMSATS University Islamabad, Lahore Campus, Lahore, Pakistan}\\
		{\small E-mail: drimranahmed@cuilahore.edu.pk}}
	\date{}
	\maketitle
	\hrule
\begin{abstract}
We first construct the total simplicial complex (TSC) of a finite simple graph $G$ in order to generalize the total graph $T(G)$. We show that $\Delta_T(G)$ is not Cohen-Macaulay (CM) in general. For a connected graph $G$, we prove that the TSC is Buchsbaum. We demonstrate that the vanishing of first homology group of TSC associated to a connected graph $G$ is both a necessary and sufficient condition for it to be CM. We find the primary decomposition of the TSC associated to a family of friendship graphs $F_{5n+1}$ and prove it to be CM.
\end{abstract}
\hrule


\section{Introduction}

In the mid-1970s, the idea of Cohen-Macaulay (CM) complexes appeared from the work of Hochster, Reisner, Stanley \cite{Ho,Re,S}, quickly becoming the main subject of an appealing and richly connected new field of mathematics at the nexus of combinatorics, commutative algebra, and topology. Stanley, as the primary architect of these developments, has made significant contributions over a long period of time. The CM complex theory has vast applications to various concepts in combinatorics, for instance matroids, polytope boundaries, geometric lattices, and intersection lattices of hyperplane arrangements, etc.

The simplicial complex $\Delta$ has immense applications in commutative algebra and algebraic topology. Due to several combinatorial and topological characteristics  of the connectivity of simplicial complexes, insightful research in topology has long been done \cite{AM,F,MABY,MAL}. Commutative algebra properties like CM and Buchsbaum are transformed into simplicial complexes via the Stanley-Reisner ring $K[\Delta]$ with $K$ a field. A commutative ring is said to be Cohen-Macaulay (CM) if its depth and Krull dimension are the same. A simplicial complex $\Delta$ is said to be CM over $K$ if $K[\Delta]$ is CM. 

The CM is a topological property. Combinatorial and geometric methods can be used to investigate their homology groups. In many fields of bioinformatics and image processing, the homology has applications \cite{C}. Reisner provided a comprehensive characterization of such complexes in 1974. In order for a simplicial complex $\Delta$ to be CM over $K$, all the reduced homology groups for the links of faces $\sigma\in\Delta$ with coefficients in $K$ must be zero, with the exception of the top-dimensional one \cite{Re}. The set of all faces $\tau\in\Delta$ makes up the link of a face $\sigma$ in $\Delta$ if $\tau\cup\sigma\in\Delta$ and $\tau\cap\sigma=\emptyset$. Every CM simplicial complex is unmixed \cite{F1}.

For a non-negative integer $t$, Haghighi, Yassemi, and Zaare-Nahandi discovered the class of CM simplicial complexes in codimension $t$, denoted by $CM_t$ \cite{Ha2}. The CM and Buchsbaum are two properties for simplicial complexes that this class generalizes.  The $CM_t$ characteristic is a topological invariant.

A graph can be considered as a one dimensional simplicial complex. The total simplicial complex (TSC) $\Delta_T(G)$ is a generalization of total graph $T(G)$ for a finite simple graph $G$ (see Definition \ref{d1}). The TSC contains line simplicial complex (LSC) \cite{AM} and Gallai simplicial complex (GSC) \cite{MAL} as subcomplexes. A total graph $T(G)$ represents adjacencies between edges and vertices of $G$. Behzad and Chartrand \cite{BC} were the first to propose the concept of total graph. Numerous authors \cite{B2,B3,B1,BR2,BR1,GSG} have conducted extensive research on the characteristics of total graphs.

We show that $\Delta_T(G)$ is not CM in general, see Example \ref{e1}. We demonstrate in Theorem \ref{t2} that the TSC of a connected graph $G$ is Buchsbaum.  In Theorem \ref{t3}, we establish that the vanishing of first homology group of $\Delta_T(G)$ associated to a connected graph $G$ is both a necessary and a sufficient condition for TSC to be CM. Finally, we find the primary decomposition of $\Delta_T(F_{5n+1})$ associated to a family of friendship graphs $F_{5n+1}$ and prove it to be CM.

\section{Basic setup}

We refer to \cite{F1,F,S1} for introductory definitions and background information on simplicial complexes. For basic definitions on simplicial homology, we refer to \cite{R}.

We consider the family of subsets over $[m]=\{1,\ldots,m\}$ that forms the simplicial complex $\Delta$  with $\tau\subset \sigma\in \Delta$ implies that $\tau\in\Delta$. The members of $\Delta$ are considered as faces of $\Delta$, and any face $\sigma\in\Delta$ has a dimension, defined as $\dim\sigma=|\sigma|-1$, where $|\sigma|$ is the number denoting how many vertices there are in $\sigma$. The count of $k$-dimensional faces in $\Delta$, denoted by $\alpha_k$,  determines the $f$-vector $(\alpha_0,\alpha_1,\ldots,\alpha_k,\ldots)$ of $\Delta$. 

For a simplicial complex $\Delta$, the maximal faces under inclusion are referred as its facets. If $\{\sigma_1,\ldots ,\sigma_h\}$ is the set of all the facets of $\Delta$, then $\Delta=<\sigma_1,\ldots,\sigma_h>$. A subset of $\Delta$ is said to be the subcomplex of $\Delta$ if it is a simplicial complex as a whole. A simplicial complex $\Delta$ is known as pure if all facets of $\Delta$ have the same dimension. We represent the dimension of $\Delta$ by $\dim\Delta=max\{\dim \sigma\ |\ \sigma\in\Delta\}$.

A simplicial complex $\Delta$ is termed to be connected if it has a succession of facets
$\sigma_0=\sigma,\ldots,\sigma_i=\bar{\sigma}$ such that $\sigma_j\cap \sigma_{j+1}\neq
\emptyset$ with $j=0,\ldots,i-1$ for any two of its facets $\sigma$ and $\bar{\sigma}$. If a simplicial complex $\Delta$ is not connected, it is said to be disconnected.

We consider $\Delta$ a simplicial complex on the vertex set $[m]=\{y_1,\ldots,y_m\}$ with facets $\sigma_1,\ldots,\sigma_h$. An ideal generated by all square-free monomials associated to non-faces of $\Delta$ is known as Stanley-Reisner ideal $I_{\mathcal{N}}(\Delta)$. The Stanley-Reisner ring $K[\Delta]$ is acquired from polynomial ring $K[y_1,\ldots,y_m]$ via quotienting out the Stanley-Reisner ideal. 

\begin{Definition}\label{dd}
Let $[m]=\{1,\ldots,m\}$ be a vertex set of a simplicial complex $\Delta$. A vertex cover for $\Delta$ is a collection $C$ of $[m]$ with the condition that there is a vertex $v \in C$ such that  $v\in \sigma_j$ for every facet $\sigma_j$. A subset $C$ of $[m]$ that is a vertex cover, and for which no proper subset of $C$ is a vertex cover for $\Delta$, is a minimal vertex cover of $\Delta$.  If every minimal vertex cover of $\Delta$ has the same cardinality, the complex is said to be unmixed.
\end{Definition}

The sequence of free abelian groups and homomorphisms
\begin{center}
$\ldots \rightarrow C_{r}(\Delta) ^{
\underrightarrow{\partial_{r}}}\ldots ^{\underrightarrow{\partial_{2}}}C_{1}({\Delta})^{
\underrightarrow{\partial_{1}}}C_{0}({\Delta})^{\underrightarrow{\partial_{0}}}0$,
\end{center}
is called a chain complex of $\Delta$ with $\partial_{i}\circ\partial_{i+1}=0$, denoted by $C_{\ast}(\Delta)$. For each $r\geq 0$, we define $C_r(\Delta)$ as a free abelian group with basis containing all $r$-dimensional faces in $\Delta$. The elements of $C_r(\Delta)$ are called $r$-chains in $\Delta$. The homomorphisms $\partial_r$ are known as
boundary operators and they are defined by
\begin{equation}\label{eq1}
\partial_{r}(\sigma_{ i_0,\ldots,i_r})=\sum_{j=0}^{r}(-1)^{j}\sigma_{i_0,\ldots,\hat{i_j},\ldots,i_r},
\end{equation}
where $\hat{i_j}$ denotes the vertex $i_j$ that will be removed from the array. The groups of $r$-cycles and $r$-boundaries, $\mathcal{Z}_r(\Delta)$ and $\mathcal{B}_r(\Delta)$, are expressed as $Ker\ \partial_{r}$ and  $Im\ \partial_{r+1}$, respectively. For each $r\geq 0$, we define the $r$-th homology group of a simplicial complex $\Delta$ by
\begin{equation}\label{eq2}
H_{r}(\Delta)=\mathcal{Z}_r(\Delta)/\mathcal{B}_r(\Delta).
\end{equation}
The $r$-th Betti number, represented by $\beta_r$, is written as the rank of $H_{r}(\Delta)$.

The homology groups of the augmented chain complex
$$\ldots \rightarrow C_{r}(\Delta) ^{\underrightarrow{\partial_{r}}}\ldots
^{\underrightarrow{\partial_{2}}}C_{1}({\Delta})^{
\underrightarrow{\partial_{1}}}C_{0}({\Delta})^{\underrightarrow
{\varepsilon}}{K}\rightarrow 0,$$
where $\varepsilon({\Sigma}_{i}(-1)^i{v_i})={\Sigma}_i(-1)^i$, are known as reduced homology groups $\tilde{H}_r(\Delta)$ such that $\Delta$ is a simplicial complex. We consider $\tilde{H}_0(\Delta)=Ker\ \varepsilon/Im\ \partial_1$ and $\tilde{H}_r(\Delta)=\mathcal{Z}_r(\Delta)/\mathcal{B}_r(\Delta)$ for positive $r$. That is, $H_0(\Delta)=\tilde{H}_0(\Delta)\oplus {K}$ and $H_r(\Delta)=\tilde{H}_r(\Delta)$ for $r>0$.

The following topological characterization for a CM simplicial complex is due
to Reisner \cite{Re}.

\begin{Theorem}\cite{Re}\label{t} The simplicial complex $\Delta$ is
CM over a field $K$ if and only if for all $\sigma\in
\Delta$ and $ r< dim\ (link_{\Delta}(\sigma)),\ \widetilde{H}_r\,(link_{\Delta}(\sigma); K) = 0,$
where
$link_{\Delta}(\sigma)=\{\tau\in\Delta\ |\ \tau\cap\sigma=\emptyset \mbox{ and } \tau\cup\sigma\in\Delta\}$
is a link of $\Delta$ at $\sigma$ and $\widetilde{H}_r\ (link_{\Delta}(\sigma); K)$ the $r$-th reduced homology group of $link_{\Delta}(\sigma)$ over $K$.
\end{Theorem}
The $CM_t$ simplicial complex was introduced in \cite{Ha2}.
\begin{Definition}\cite{Ha2}
Assume that $\Delta$ is a simplicial complex of dimension $d-1$. Consider $t$ an integer such that $0\leq t\leq d-1$. When $\Delta$ is
pure and $link_{\Delta}(\sigma)$ is CM over over a field $K$ for any $\sigma$ in $\Delta$ with $\#\sigma\geq t$, $\Delta$ is then known as $CM_t$ over $K$.
\end{Definition}

According to Reisner \cite{Re} and Schenzel \cite{Sc}, $CM_0$ is the CM property, while $CM_1$ the Buchsbaum.

Haghighi, Yassemi and Zaare-Nahandi established the following topological characterization of a $CM_t$ simplicial complex \cite{Ha2}.

\begin{Theorem}\cite{Ha2}\label{tt}
The simplicial complex $\Delta$ of dimension $d-1$ is $CM_t$ over a field $K$ if and only 
it is pure and for all $\sigma\in\Delta$, $\widetilde{H}_r\,(link_{\Delta}(\sigma); K)=0$  such that $\#\sigma\geq t$ and $r<d-\#\sigma-1$.
\end{Theorem}


\section{CM Property of Total Simplicial Complexes}

Behzad and Chartrand's description for total graph $T(G)$ of a finite simple graph $G$ serves as the central aspect of this work \cite{BC}.

\begin{Definition}\cite{BC}
We suppose that  $G$ is a finite simple graph containing vertex set $[m]$ and edge set $E(G)$. For a total graph $T(G)$ of $G$ with vertex set $[m]\cup E(G)$, two vertices are adjacent in $T(G)$ if and only if they are adjacent or incident in $G$.
\end{Definition}

For the structural analysis of TSC, we define total indices corresponding to a finite simple graph $G$.

\begin{Definition}\label{d0}
We assume that  $G$ is a finite simple graph containing vertex set $[m]$ and edge set $E(G)$. The set of all total indices $\Upsilon_{T}(G)$ of $G$ is defined as a collection of subsets of $[m]\cup E(G)$ with $\sigma_{p,q,r}=\{p,q,r\}\in \Upsilon_{T}(G)$ if
\noindent (i) $p, q, r$ are adjacent vertices in $G$; or
\noindent (ii) $p, q, r$ are adjacent edges in $G$; or
\noindent (iii) $p$ and $q$ are adjacent vertices and $r$ is an incident edge to $p$ or $q$; or
\noindent (iv) $p$ and $q$ are adjacent edges and $r$ is an incident vertex to $p$ or $q$.
Moreover,  $\sigma_p=\{p\}\in\Upsilon_{T}(G)$ if $p$ is an isolated vertex in $G$.
\end{Definition}
\begin{Definition}\label{d1}
Suppose that  $G$ is a finite simple graph having vertex set $[m]$ and edge set $E(G)$. The total simplicial complex (TSC) $\Delta_{T}(G)$ of $G$ on vertex set $[m]\cup E(G)$ is defined by
$$\Delta_{T}(G)=\langle \sigma |\ \sigma \in\Upsilon_{T}(G)\rangle.$$
\end{Definition}

\begin{Remark}
A total index is said to be a line index if the condition (i) of Definition \ref{d0} holds. The line index is said to be a Gallai index if the incident edges $e_{i,j}$ and $e_{j,k}$ of $G$ do not span a triangle in $G$. Therefore, the TSC contains LSC \cite{AM} and GSC \cite{MAL} as subcomplexes.
\end{Remark}

One can assume that the TSC is built on the vertex set $V(G)\cup E(G)=[m]\cup E(G)=[N]$ by relabeling the edges $E(G)$ with \{m+1,\ldots,N\}. 

On the basis of connectivity of a finite simple graph $G$, we develop first the criteria for connectivity of the TSC associated to a finite simple graph $G$.
\begin{Lemma}\label{l1}
Assume that  $G$ is a finite simple graph having vertex set $[m]$ and edge set $E(G)$. Then, the link of the TSC associated to $G$ at $\emptyset$ is connected if and only if $G$ is connected.
\end{Lemma}
\proof
One can easily see that $link_{\Delta_T(G)}(\emptyset)=\Delta_T(G)$. Therefore, we need to establish the necessary and sufficient condition for $\Delta_T(G)$ to be connected is $G$ to be connected.

We consider $G$ a finite simple graph containing vertex set $[m]$ and edge set $E(G)$ such that $[m]\cup E(G)=[N]$. For  $N\leq3$, the result is trivial  and can be treated into separate cases. So, from now on we take $N\geq4$. 

First, we prove that $\Delta_{T}(G)$ is connected if $G$ is connected. We prove it by contradiction. We assume that the TSC of $G$ is not connected. Then, as long as no face of $\Delta_T(G)$ includes vertices in both $V_1$ and $V_2$, the vertex set $[N]$ of $\Delta_T(G)$ can be expressed as disjoint union of two non-empty subsets $V_1$ and $V_2$ of $[N]$.  Since $G$ is connected and $N\geq 4$, therefore $\Delta_T(G)$ is a pure simplicial complex having dimension $2$, see Definition \ref{d1}. Consequently, there are at least two facets $\sigma_{p,q,r}, \sigma_{s,t,u}\in\Delta_T(G)$ such that $p, q, r\in V_1\subset [N]$ and $s, t, u\in V_2\subset [N]$. By Definition \ref{d1}, the vertices $p, q, r$ (respectively $s, t, u$) are adjacent vertices of total graph $T(G)$. Thus, $p, q, r$ (respectively $s, t, u$) in $G$ must satisfy one of the following conditions:\\
\noindent (i) $p, q, r$ (respectively $s, t, u$) are adjacent vertices in $G$;
\noindent (ii) $p, q, r$ (respectively $s, t, u$) are adjacent edges in $G$;
\noindent (iii) $p$ and $q$ (respectively $s$ and $t$) are adjacent vertices and $r$ (respectively $u$) is an incident edge to $p$ or $q$ (respectively $s$ or $t$) in $G$;
\noindent (iv) $p$ and $q$ (respectively $s$ and $t$) are adjacent edges and $r$ (respectively $u$) is an incident vertex to $p$ or $q$ (respectively $s$ or $t$) in $G$.

But there is no path in $G$ joining $p$, $q$ or $r$ with $s$, $t$ or $u$ as end points. It implies that $G$ is disconnected, a contradiction. Hence the condition is sufficient.

Now, we prove that if $\Delta_T(G)$ is connected then $G$ is connected. On contrary, we suppose that $G$ is not connected. Therefore, there exist vertices $i, j\in G$ such that there is no path in $G$ joining $i$ and $j$ as end points. Consequently, there are at least two facets, say $\sigma_{p,q,r}$ and $\sigma_{s,t,u}$, in $\Delta_T(G)$ with two non-empty disjoint  subsets $ V_1\ni p, q, r$ and $V_2\ni s, t, u$ of $[N]$ such that there is no succession of facets joining $\sigma_{p,q,r}$ and $\sigma_{s,t,u}$, see Definition \ref{d1}. Thus, $\Delta_{T}(G)$ is disconnected, a contradiction. Therefore, the condition is necessary.
\endproof

We demonstrate that for the TSC of a simple graph $G$, the CM property is not universally true.

\begin{Remark}\label{r1}
Consider a simple graph  $G$. If $G$ is not connected, then the TSC of $G$ is not $CM_t$ with $t=0,1$, see Theorem \ref{tt} and Lemma \ref{l1}.
\end{Remark}

The following example suggests us that the TSC of a connected graph $G$ is not CM in general.

\begin{Example}\label{e1}
Every CM simplicial complex, as far as we are aware, is unmixed \cite{F1}. We intend to find a connected TSC, which is not unmixed, see Remark \ref{r1}. The TSC of a connected graph is connected, see Lemma \ref{l1}. Therefore, we want to find a connected graph for which the associated TSC is not unmixed. Let's consider the graph $C_{4,2}$ having $2$ cycles of length $4$ with two edges common, as shown in Figure ~\ref{fig:f2}. The TSC of $C_{4,2}$ is given by\\
$\Delta_T(C_{4,2})=\langle \sigma_{1,2,3}, \sigma_{2,3,4}, \sigma_{3,4,5}, \sigma_{4,5,6}, \sigma_{5,6,7},
\sigma_{6,7,8}, \sigma_{1,9,10}, \sigma_{9,10,11}, \sigma_{5,10,11}, \sigma_{1,3,4},\\ \sigma_{1,3,8}, \sigma_{1,3,9},
\sigma_{1,7,8},\sigma_{1,2,7}, \sigma_{1,7,9}, \sigma_{1,6,7}, \sigma_{1,2,10}, \sigma_{1,8,10}, \sigma_{1,10,11},
\sigma_{2,3,5}, \sigma_{3,5,6}, \sigma_{4,5,7},\\ \sigma_{5,7,8}, \sigma_{4,5,10}, \sigma_{5,6,10}, \sigma_{5,9,10},
\sigma_{1,2,4}, \sigma_{2,4,5}, \sigma_{1,2,9}, \sigma_{2,3,9}, \sigma_{2,9,10}, \sigma_{1,2,8}, \sigma_{2,3,8},
\sigma_{2,7,8},\\ \sigma_{1,8,9}, \sigma_{8,9,10}, \sigma_{7,8,9}, \sigma_{1,9,11}, \sigma_{5,9,11}, \sigma_{3,4,6},
\sigma_{4,6,7}, \sigma_{3,4,11},\sigma_{4,5,11}, \sigma_{4,10,11}, \sigma_{5,6,11},\\ \sigma_{6,10,11}, \sigma_{6,7,11},
\sigma_{5,6,8}, \sigma_{1,6,8}, \sigma_{1,3,5},\sigma_{1,5,7}, \sigma_{1,5,10}, \sigma_{1,3,10}, \sigma_{1,7,10},
\sigma_{3,5,7}, \sigma_{5,7,10}, \sigma_{3,5,10},\\ \sigma_{2,8,9},\sigma_{2,4,9}, \sigma_{2,9,11}, \sigma_{2,4,6},
\sigma_{2,4,8}, \sigma_{2,6,8}, \sigma_{2,4,11}, \sigma_{4,6,11}, \sigma_{4,6,8},\sigma_{4,9,11}, \sigma_{6,9,11},
\sigma_{6,8,11},\\ \sigma_{6,8,9}, \sigma_{8,9,11}, \sigma_{3,5,11}, \sigma_{5,7,11}\rangle$, see Definition \ref{d1}.

\begin{figure}[ht]
	\centering
	\includegraphics[scale=0.5]{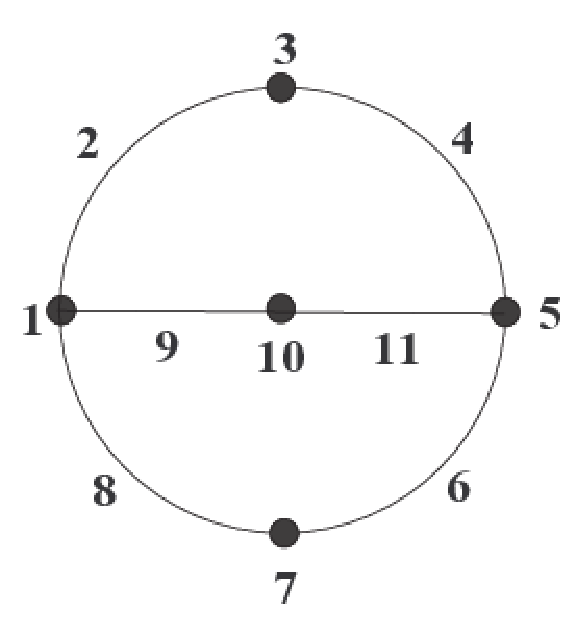}
	\caption{$C_{4,2}$}
	\label{fig:f2}
\end{figure}

One can see that the vertex covers $\{1,4,5,6,8,9\}$ and
$\{1,2,4,5,6,8,10\}$ are the  minimal vertex covers of $\Delta_T(C_{4,2})$ that don't have same cardinalities. Thus, the TSC of $C_{4,2}$ is not unmixed and therefore not Cohen-Macaulay.
\end{Example}

\begin{Lemma}\label{l2}
Let $\Delta_{T}(G)$ be TSC of a finite simple graph $G$. The link of $\Delta_{T}(G)$ at $\{i\}$ is connected for every $i\in\Delta_{T}(G)$.
\end{Lemma}

\proof
If $G$ is a disconnected graph and it has $n$ connected components, say $G_1, \ldots, G_n$. Then the TSC $\Delta_T(G)$ is a disconnected simplicial complex and it consists of $n$ connected components $\Delta_T(G_1),\ldots,$ $\Delta_T(G_n)$, see Lemma \ref{l1} and Definition \ref{d1}. It is enough to prove the result for a connected component of $\Delta_T(G)$. Therefore we can assume that $G$ is a connected graph.    

We establish this result by contradiction. Suppose that $G$ is a connected graph. On contrary, we assume that the link $link_{\Delta_T(G)}(\{i\})$ of $\Delta_{T}(G)$ at $\{i\}$ is not connected for some $i\in\Delta_{T}(G)$. Since $G$ is a connected graph, therefore the TSC associated to $G$ is connected by Lemma \ref{l1}.  Consequently, for any pair of vertices $r$ and $s$ of $\Delta_{T}(G)$ there exists a sequence of facets $\sigma_1,\ldots,\sigma_q$ of $\Delta_{T}(G)$ with $r\in \sigma_1$ and $s\in \sigma_q$ such that $\sigma_q\cap \sigma_{q+1}\neq\emptyset$. It implies that for any pair of vertices or edges, $r$ and $s$, of $G$ there exists a sequence of adjacent/incident vertices or edges of $G$ joining $r$ and $s$ as end points. 

Note that $\Delta_{T}(G)$ is a simplicial complex of dimension $2$, see Definition \ref{d1}. Therefore, one can easily see that $link_{\Delta_T(G)}(\{i\})$ is a $1$-dimensional subcomplex of $\Delta_{T}(G)$. It is a disconnected graph on the vertex set $V(G)\cup E(G)$ for some $i$ by our assumption. So, there exist edges $\{j_1,j_2\}$ and $\{k_1,k_2\}$ of $link_{\Delta_T(G)}(\{i\})$ such that there is no path in $link_{\Delta_T(G)}(\{i\})$ joining $j_h$ and $k_l$ as end points with $h, l=1, 2$. It implies that $j_h$ and $k_l$ are vertices or edges of $G$ such that there is no sequence of adjacent/incident vertices or edges of $G$ joining $j_h$ and $k_l$ as end points, a contradiction. Hence the result.
\endproof

\begin{Theorem}\label{t2}
We consider $G$ a connected graph. The TSC associated to $G$ is Buchsbaum.
\end{Theorem}
\proof
By Theorem \ref{tt}, for being $CM_1$ it is necessary and sufficient that\\ $\widetilde{H}_k\,(link_{\Delta_T(G)}(\sigma); K)=0$ for all $\sigma\in\Delta_T(G)$ for which $\#\sigma\geq 1$ and $k<d-\#\sigma-1=3-\#\sigma-1$ (since for each facet $\sigma\in\Delta_T(G)$, $\#\sigma=3$). If $\#\sigma=2$, then $k<3-\#\sigma-1=3-2-1=0$, so the homology groups $\widetilde{H}_k\,(link_{\Delta_T(G)}(\sigma); K)$ are already zero. Therefore, it is enough to consider the link of single element faces. In this case, it is enough to prove that the link $link_{\Delta_T(G)}(\{i\})$ of $\Delta_{T}(G)$ at $\{i\}$ is connected for every $i\in V(G)\cup E(G)$, which has been established in Lemma \ref{l2}. Hence proved.   
\endproof

\begin{Theorem}\label{t3}
Consider $G$ a connected graph. The necessary and sufficient condition for total simplicial complex $\Delta_{T}(G)$ to be CM is vanishing of its first homology group.
\end{Theorem}
\proof
From Theorem \ref{tt}, for being $CM_0$ it is necessary and sufficient that
$\widetilde{H}_k\,(link_{\Delta_T(G)}(\sigma); K)=0$ for all $\sigma\in\Delta_T(G)$ for which $\#\sigma\geq 0$ and $k<d-\#\sigma-1=3-\#\sigma-1$. 

If $\#\sigma=2$, then $k<3-\#\sigma-1=3-2-1=0$, so the homology groups $\widetilde{H}_k\,(link_{\Delta_T(G)}(\sigma); K)$ are already zero. 

If $\#\sigma=1$, then $k<3-\#\sigma-1=3-1-1=1$. In this case, it is enough to prove that the link $link_{\Delta_T(G)}(\{i\})$ is connected for every $i\in V(G)\cup E(G)$, which has been proved in Lemma \ref{l2}.

If $\#\sigma=0$, then $k<3-\#\sigma-1=2$. In this case, we want to prove that $\widetilde{H}_k\,(link_{\Delta_T(G)}(\emptyset); K)=0$ with $k=0, 1$. Since $G$ is a connected graph, therefore the link $link_{\Delta_T(G)}(\emptyset)$ is connected by Lemma \ref{l1}. Thus, $\widetilde{H}_0\,(link_{\Delta_T(G)}(\emptyset); K)=0$. By hypothesis, $\widetilde{H}_1\,(link_{\Delta_T(G)}(\emptyset); K)=0$. Hence the result.
\endproof

\section{CM Property of $\Delta_T(F_{5n+1})$ and Associated Primes of the Facet Ideal $I_{\mathcal{F}}(\Delta_T(F_{5n+1}))$}

We fix the labeling of vertices and edges of a family of friendship graphs $F_{5n+1}$ with  $2n+1$ vertices and $3n$ edges as shown in Figure ~\ref{fig:f1}.
\begin{figure}[ht]
\centering
\includegraphics[scale=1.4]{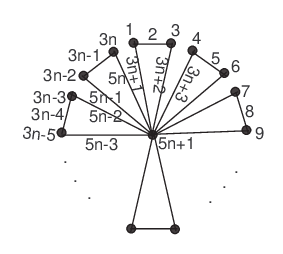}
\caption{$F_{5n+1}$}
\label{fig:f1}
\end{figure}

The following result provides us a combinatorial build-up of TSC associated to a family of friendship graphs $F_{5n+1}$.

\begin{Lemma}\label{l3}
Let $F_{5n+1}$ be a family of friendship graphs having $2n+1$ vertices and $3n$ edges. The TSC $\Delta_T(F_{5n+1})$ of $F_{5n+1}$ is generated by the following facets:
\begin{enumerate}
 \item $\sigma_{i,i+1,i+2}$ with $i=1,4,\ldots,3n-2;$

\item $\sigma_{i,j,5n+1}$  with $i = 1,4,\ldots,3n-2$, $j =3n+1,\ldots,5n$ and $i =3,6,\ldots,3n$, $j=3n+1,\ldots,5n;$

\item $\sigma_{3n+i,3n+j,3n+k}$ with $i= 1,\ldots,2n-2$,\\
$j= i+1,\ldots,2n-1$, $k= j+1,\ldots,2n;$

\item $\sigma_{3n+i,3n+j,5n+1}$ with $i = 1,\ldots,2n-1$,
$j = i+1,\ldots,2n;$

\item $\sigma_{i,i+1,3n+k}$ and $\sigma_{i,i+1,3n+k+1}$ for corresponding values of $i$ and $k$ in the sequences
$i = 1,4,\ldots,3n-2$, $k =1,3,\ldots,2n-1$ and $i=2,5,\ldots,3n-1$, $k=1,3,\ldots,2n-1;$

\item $\sigma_{i,i+1,5n+1}$ with $i = 1,4,\ldots,3n-2$,
$i=2,5,\ldots,3n-1;$

\item $\sigma_{i,j,5n+1},\ \sigma_{i,j-1,5n+1};\ \sigma_{i,j,5n+1},\ \sigma_{i,j+2,5n+1}; \sigma_{i,3n,5n+1}$ with $i = 1,4,\ldots,3n-5,\ j=i+3,i+6,\ldots,3n-2;\ i=3,6,\ldots,3n-3$,
$j=i+1,i+4,\ldots,3n-2;\ i=1,4,\ldots,3n-2$ (respectively)$;$

\item $\sigma_{i,3n+j,5n+1}$ and $\sigma_{i,3n+j+1,5n+1}$ for corresponding values of $i=2,5,\ldots,3n-1$ and
$j=1,3,\ldots,2n-1;$

\item $\sigma_{i,3n+j,3n+k}$ for corresponding values of $i$ and $j$ in $i=2,5,\ldots,3n-1,\ j=1,3,\ldots,2n-1$,\\
$k=j+1,\ldots,2n$ and $i=2,5,\ldots,3n-4$,\\
$j=2,4,\ldots,2n-2,\ k=j+1,\ldots,2n;$

\item $\sigma_{3n-i,5n-k,5n-j}$ for corresponding values of $i$ and $j$ in $i=1,4,\ldots,3n-5,\ j=0,2,\ldots,2n-4,\ k=j+2,\ldots,2n-1$ and $i=1,4,\ldots,3n-5,\ j=1,3,\ldots,2n-3,\ k=j+1,\ldots,2n-1;$

\item $\sigma_{i,i+2,3n+j}$ and $\sigma_{i,i+2,3n+j+1}$ for corresponding values of $i=1,4,\ldots,3n-2$ and
$j=1,3,\ldots,2n-1;$

\item $\sigma_{i,3n+j,3n+k}$ for corresponding values of $i$ and $j$ in $i=1,4,\ldots,3n-2,\ j=1,3,\ldots,2n-1$,
$k=j+1,\ldots,2n$ and $i=3,6,\ldots,3n-3$,
$j=2,4,\ldots,2n-2,\ k=j+1,\ldots,2n;$

\item $\sigma_{3n-i,5n-k,5n-j}$ for corresponding values of $i$ and $j$ in $i=2,5,\ldots,3n-4,\ j=1,3,\ldots,2n-3,\ k=j+1,\ldots,2n-1$ and $i=0,3,\ldots,3n-3,\ j=0,2,\ldots,2n-2,\ k=j+1,\ldots,2n-1.$
\end{enumerate}
\end{Lemma}
\proof
One can see from Definition \ref{d1} that the TSC of a family of friendship graphs $F_{5n+1}$ is generated by all the facets in lists (1) to (13), given above.
\endproof

Now, we give formulation for $f$-vector of TSC associated to $F_{5n+1}$.

\begin{Lemma}\label{l4}
Let $\Delta_{T}(F_{5n+1})$ be TSC associated to a family of friendship graphs $F_{5n+1}$. Then, the $f$-vector of $\Delta_{T}(F_{5n+1})$ is given by
$$(\alpha_0,\alpha_1,\alpha_2)=(5n+1, 10n^{2}+5n, \frac{4n^3+42n^2+14n}{3}).$$
\end{Lemma}
\proof
Since $F_{5n+1}$ contains $2n+1$ vertices and $3n$ edges, therefore $\alpha_0=5n+1$. We compute  $\alpha_2$ into the following steps.
\begin{enumerate}
\item $|\{\{i,i+1,i+2\}:\ i=1,4,\ldots,3n-2\}|=n$;

\item $|\{\{i,j,5n+1\}:\ i = 1,4,\ldots,3n-2,\ 
j =3n+1,\ldots,5n\ \mbox{and}\ i =3,6,\ldots,3n,\ 
j=3n+1,\ldots,5n\}|=2n+\ldots+2n+2n+\ldots+2n
=2n^2+2n^2=4n^2$;

\item $|\{\{3n+i,3n+j,3n+k\}:\ i= 1,\ldots,2n-2,\ j= i+1,\ldots,2n-1,\ k= j+1,\ldots,2n\}|
=\frac{(2n-2)(2n-1)}{2}+\frac{(2n-3)(2n-2)}{2}+\frac{(2n-4)(2n-3)}{2}+\frac{(2n-5)(2n-4)}{2}+\ldots+
\frac{4.5}{2}+\frac{3.4}{2}+\frac{2.3}{2}+1
=1.1+1.3+2.3+2.5+3.5+3.7+\ldots+(n-1)(2n-3)+(n-1)(2n-1)=1.1+2.3+3.5+\ldots+(n-1)(2n-3)+
1.3+2.5+3.7+\ldots+(n-1)(2n-1)
=\sum_{k=1}^{n-1}k(2k-1)+\sum_{k=1}^{n-1}k(2k+1)
=4\sum_{k=1}^{n-1}k^2=\frac{2n(n-1)(2n-1)}{3}$;

\item $|\{\{3n+i,3n+j,5n+1\}:\ i = 1,\ldots,2n-1\ \mbox{and}\ 
j = i+1,\ldots,2n\}|=2n-1+\ldots+1
=n(2n-1)$;

\item $|\{\{i,i+1,3n+k\}:\ \mbox{for corresponding values of}\ i\ \mbox{and}\ k\ \mbox{in the sequences}\ i = 1,4,\ldots,3n-2,\ 
k =1,3,\ldots,2n-1\ \mbox{and}\ i=2,5,\ldots,3n-1,\ 
k=1,3,\ldots,2n-1\}|+|\{\{i,i+1,3n+k+1\}:\ \mbox{for corresponding values of}\ i\ \mbox{and}\ k\ \mbox{in the sequences}\ i = 1,4,\ldots,3n-2,\ 
k =1,3,\ldots,2n-1\ \mbox{and}\ i=2,5,\ldots,3n-1,\ 
k=1,3,\ldots,2n-1\}|
=2n+2n=4n$;

\item $|\{\{i,i+1,5n+1\}:\ i = 1,4,\ldots,3n-2; i=2,5,\ldots,3n-1\}|=n+n=2n$;

\item $(|\{\{i,j,5n+1\}:\ i = 1,4,\ldots,3n-5,\ j=i+3,i+6,\ldots,3n-2\}|+|\{\{i,j-1,5n+1\}:\ i = 1,4,\ldots,3n-5,\ j=i+3,i+6,\ldots,3n-2\}|)+\\(|\{\{i,j,5n+1\}:\ i=3,6,\ldots,3n-3,\  j=i+1,i+4,\ldots,3n-2\}|+|\{\{i,j+2,5n+1\}:\ i=3,6,\ldots,3n-3,\ j=i+1,i+4,\ldots,3n-2\}|)+\\|\{\{i,3n,5n+1\}:\ i=1,4,\ldots,3n-2\}|=(n-1+\ldots+1+n-1+\ldots+1)+
(n-1+\ldots+1+n-1+\ldots+1)+n
=(n-1)n+(n-1)n+n=n(2n-1)$;

\item $|\{\{i,3n+j,5n+1\}:\ \mbox{for corresponding values of}\ 
i=2,5,\ldots,3n-1\ \mbox{and}\ j=1,3,\ldots,2n-1\}|+|\{\{i,3n+j+1,5n+1\}:\ \mbox{for corresponding values of}\ 
i=2,5,\ldots,3n-1\ \mbox{and}\ j=1,3,\ldots,2n-1\}|
=n+n=2n$;

\item $|\{\{i,3n+j,3n+k\}:\ \mbox{for corresponding values of}\ i\ \mbox{and}\ j\ \mbox{in}\ i=2,5,\ldots,3n-1,\ j=1,3,\ldots,2n-1,\ k=j+1,\ldots,2n\ \mbox{and}\ i=2,5,\ldots,3n-4,\  j=2,4,\ldots,2n-2,\ k=j+1,\ldots,2n\}|=2n-1+2n-3+\ldots+1+2n-2+2n-4+\ldots+2=n(2n-1)$;

\item $|\{\{3n-i,5n-k,5n-j\}:\ \mbox{for corresponding values of}\ i\ \mbox{and}\ j\ \mbox{in}\ i=1,4,\ldots,3n-5,\ j=0,2,\ldots,2n-4,\ 
k=j+2,\ldots,2n-1\ \mbox{and}\ i=1,4,\ldots,3n-5,\ j=1,3,\ldots,2n-3,\ k=j+1,\ldots,2n-1\}|=2n-2+2n-4+\ldots+2+2n-2+2n-4+\ldots+2=2(1+\ldots+n-1+1+\ldots+n-1)=n(n-1)+n(n-1)
=2n(n-1)$;

\item $|\{\{i,i+2,3n+j\}:\ \mbox{for corresponding values of}\ 
i=1,4,\ldots,3n-2\ \mbox{and}\ j=1,3,\ldots,2n-1\}|+|\{\{i,i+2,3n+j+1\}:\ \mbox{for corresponding values of}\ 
i=1,4,\ldots,3n-2\ \mbox{and}\ j=1,3,\ldots,2n-1\}|
=n+n=2n$;

\item $|\{\{i,3n+j,3n+k\}:\ \mbox{for corresponding values of}\ i\ \mbox{and}\ j\ \mbox{in}\ i=1,4,\ldots,3n-2,\ j=1,3,\ldots,2n-1,\ 
k=j+1,\ldots,2n\ \mbox{and}\ i=3,6,\ldots,3n-3,\ j=2,4,\ldots,2n-2,\ k=j+1,\ldots,2n\}|=2n-1+2n-3+\ldots+1+2n-2+2n-4+\ldots+2=1+\ldots+2n-1
=n(2n-1)$;

\item $|\{\{3n-i,5n-k,5n-j\}:\ \mbox{for
corresponding values of}\ i\ \mbox{and}\ j\ \mbox{in}\ i=2,5,\ldots,3n-4,\ j=1,3,\ldots, 2n-3,\ k=j+1,\ldots,2n-1\ \mbox{and}\ i=0,3,\ldots,3n-3,\ j=0,2,\ldots,2n-2,
\ k=j+1,\ldots,2n-1\}|=2n-2+2n-4+\ldots+2+2n-1+2n-3+\ldots+1=1+\ldots+2n-1=n(2n-1)$.
\end{enumerate}
Thus, $\alpha_2=\frac{4n^3+42n^2+14n}{3}$. Now, we calculate $\alpha_1$ into the following steps.
\begin{enumerate}
	\item $|\{\{i,j\}:\ i=1,\ldots,3n,\ j=3n+1,\ldots,5n+1\}|=3n(2n+1)$;
	
	\item $|\{\{i,j\}:\ i=3n+1,\ldots,5n,\ j=i+1,\ldots,5n+1\}|=2n+\ldots+1=2n^2+n$;
	
	\item $|\{\{i,i+1\},\ \{i+1,i+2\},\ \{i,i+2\}:\ i=1,4,\ldots,3n-2\}|=n+n+n=3n$;
	
	\item $|\{\{i,j\}:\ i=1,4,\ldots,3n-5,\ j=i+3,i+6,\ldots,3n-2;\ i=1,4,\ldots,3n-5,\ j=i+5,i+8,\ldots,3n;\ i=3,6,\ldots,3n-3,\ j=i+3,i+6,\ldots,3n;\ i=3,6,\ldots,3n-3,\  j=i+1,i+4,\ldots,3n-2\}|=n-1+\ldots+1+n-1+\ldots+1+n-1+\ldots+1+n-1+\ldots+1=4(1+\ldots+n-1)=2n^{2}-2n$.
\end{enumerate}
Therefore,
$\alpha_{1}=10n^{2}+5n$. Hence proved.
\endproof

We give now formulation of homology groups for the TSC of $F_{5n+1}$.

\begin{Lemma}\label{t5}
Let $\Delta_{T}(F_{5n+1})$
be the TSC of a family of friendship graphs $F_{5n+1}$. Then, the homology groups of
 $\Delta_{T}(F_{5n+1})$ can be expressed as
 \begin{equation*}
H_r(\Delta_{T}(F_{5n+1}))=\left\{
\begin{array}{ll}
K, & \hbox{if $r=0$;} \\
0, & \hbox{if $r=1$;} \\
K^{\beta_2}, & \hbox{if $r=2$;} \\
0, & \hbox{if $r\geq 3$,}
\end{array}
\right.
\end{equation*}
 where $\beta_{2}=\frac{4n^{3}+12n^{2}+14n}{3}$
 for $ n\geq 1$.
\end{Lemma}
\proof
We know that the TSC of a simple graph $G$ is a pure complex having dimension $2$. Therefore, the chain complex of TSC associated to $F_{5n+1}$ is given by
$$\ldots ^{\underrightarrow{\partial_4}}0 ^{\underrightarrow{\partial_3}} C_2(\Delta_{T}(F_{5n+1}))^{\underrightarrow{\partial_{2}}}  C_{1}(\Delta_{T}(F_{5n+1}))^{\underrightarrow{\partial_{1}}} C_{0}(\Delta_{T}(F_{5n+1}))^{\underrightarrow{\partial_{0}}} 0,$$
where $C_r(\Delta_{T}(F_{5n+1}))$ is a free abelian group generated by $r$-dimensional faces of $\Delta_{T}(F_{5n+1})$ with $r=0,1,2$ such that $n\geq 1$. The boundary homomorphism $\partial_i$ is a matrix of order $\alpha_{i-1}\times \alpha_i$ with $i=1, 2$, see Equation (\ref{eq1}).

Next, we present algorithms to find rank of matrices $\partial_i$ with $i=1, 2$, see Lemmas \ref{l3} and \ref{l4}.\\
Algorithm to Find the Rank of the Matrix $\partial_1$
\\1. Begin
\\2. Input the value of $n\geq 1$
\\3. BoundaryConditions1$=[\ ]$
\\4. $\ \ \ \ $ do $i$ from $1:+2:3*n$
\\5. $\ \ \ \ \ \ \ \ \ \ \ \ $ do $j$ from $3*n+1$ to $5*n+1$
\\6. $\ \ \ \ \ \ \ \ \ \ \ \ $ Eq1$=v[j]-v[i]$
\\7. $\ \ \ \ \ \ \ \ \ \ \ \ $ BoundaryConditions1$[i]=$Eq1
\\8. $\ \ \ \ \ \ \ \ \ \ \ \ $ end do
\\9. $\ \ \ \ $ end do
\\10. BoundaryConditions2$=[\ ]$
\\11. $\ \ \ \ $ do $i$ from $2:+3:3*n-1$
\\12. $\ \ \ \ $ Eq2$=v[5*n+1]-v[i]$
\\13. $\ \ \ \ $ BoundaryConditions2$[i]=$Eq2
\\14. $\ \ \ \ $ end do
\\15. BoundaryConditions3$=[\ ]$
\\16. $\ \ \ \ $ do $i$ from $3*n+1$ to $5*n$
\\17. $\ \ \ \ $ Eq3$=v[5*n+1]-v[i]$
\\18. $\ \ \ \ $ BoundaryConditions3$[i]=$Eq3
\\19. $\ \ \ \ $ end do
\\20. BoundaryConditions4$=[\ ]$
\\21. $\ \ \ \ $ do $i$ from $2$ to $3*n-1$
\\22. $\ \ \ \ \ \ \ \ \ \ \ \ $ do $j$ from $3*n+1:+2:5*n-1$
\\23. $\ \ \ \ \ \ \ \ \ \ \ \ $ Eq4$=v[j]-v[i]$
\\24. $\ \ \ \ \ \ \ \ \ \ \ \ $ BoundaryConditions4$[i]=$Eq4
\\25. $\ \ \ \ \ \ \ \ \ \ \ \ $ end do
\\26. $\ \ \ \ \ \ \ \ \ \ \ \ $ i=i+3
\\27. $\ \ \ \ \ \ \ \ \ \ \ \ $ break
\\28. $\ \ \ \ $ end do
\\29. BoundaryConditions5$=[\ ]$
\\30. $\ \ \ \ $ do $i$ from $2$ to $3*n-1$
\\31. $\ \ \ \ \ \ \ \ \ \ \ \ $ do $j$ from $3*n+2:+2:5*n$
\\32. $\ \ \ \ \ \ \ \ \ \ \ \ $ Eq5$=v[j]-v[i]$
\\33. $\ \ \ \ \ \ \ \ \ \ \ \ $ BoundaryConditions5$[i]=$Eq5
\\34. $\ \ \ \ \ \ \ \ \ \ \ \ $ end do
\\35. $\ \ \ \ \ \ \ \ \ \ \ \ $ i=i+3
\\36. $\ \ \ \ \ \ \ \ \ \ \ \ $ break
\\37. $\ \ \ \ $ end do
\\38. BoundaryConditions6$=[\ ]$
\\39. $\ \ \ \ $ do $i$ from $3*n+1$ to $5*n$
\\40. $\ \ \ \ \ \ \ \ \ \ \ \ $ do $j$ from $i+1$ to $5*n+1$
\\41. $\ \ \ \ \ \ \ \ \ \ \ \ $ Eq6$=v[j]-v[i]$
\\42. $\ \ \ \ \ \ \ \ \ \ \ \ $ BoundaryConditions6$[i]=$Eq6
\\43. $\ \ \ \ \ \ \ \ \ \ \ \ $ end do
\\44. $\ \ \ \ $ end do
\\45. BoundaryConditions7$=[\ ]$
\\46. $\ \ \ \ $ do $i$ from $1:+3:3*n-2$
\\47. $\ \ \ \ $ Eq7$=v[i+1]-v[i]$; $v[i+2]-v[i+1]$; $v[i+2]-v[i]$
\\48. $\ \ \ \ $ BoundaryConditions7$[i]=$Eq7
\\49. $\ \ \ \ $ end do
\\50. BoundaryConditions8$=[\ ]$
\\51. $\ \ \ \ $ do $i$ from $1:+3:3*n-5$
\\52. $\ \ \ \ \ \ \ \ \ \ \ \ $ do $j$ from $i+3:+3:3*n-2$
\\53. $\ \ \ \ \ \ \ \ \ \ \ \ $ Eq8$=v[j]-v[i]$
\\54. $\ \ \ \ \ \ \ \ \ \ \ \ $ BoundaryConditions8$[i]=$Eq8
\\55. $\ \ \ \ \ \ \ \ \ \ \ \ $ end do
\\56. $\ \ \ \ $ end do
\\57. BoundaryConditions9$=[\ ]$
\\58. $\ \ \ \ $ do $i$ from $1:+3:3*n-5$
\\59. $\ \ \ \ \ \ \ \ \ \ \ \ $ do $j$ from $i+5:+3:3*n$
\\60. $\ \ \ \ \ \ \ \ \ \ \ \ $ Eq9$=v[j]-v[i]$
\\61. $\ \ \ \ \ \ \ \ \ \ \ \ $ BoundaryConditions9$[i]=$Eq9
\\62. $\ \ \ \ \ \ \ \ \ \ \ \ $ end do
\\63. $\ \ \ \ $ end do
\\64. BoundaryConditions10$=[\ ]$
\\65. $\ \ \ \ $ do $i$ from $3:+3:3*n-3$
\\66. $\ \ \ \ \ \ \ \ \ \ \ \ $ do $j$ from $i+3:+3:3*n$
\\67. $\ \ \ \ \ \ \ \ \ \ \ \ $ Eq10$=v[j]-v[i]$
\\68. $\ \ \ \ \ \ \ \ \ \ \ \ $ BoundaryConditions10$[i]=$Eq10
\\69. $\ \ \ \ \ \ \ \ \ \ \ \ $ end do
\\70. $\ \ \ \ $ end do
\\71. BoundaryConditions11$=[\ ]$
\\72. $\ \ \ \ $ do $i$ from $3:+3:3*n-3$
\\73. $\ \ \ \ \ \ \ \ \ \ \ \ $ do $j$ from $i+1:+3:3*n-2$
\\74. $\ \ \ \ \ \ \ \ \ \ \ \ $ Eq11$=v[j]-v[i]$
\\75. $\ \ \ \ \ \ \ \ \ \ \ \ $ BoundaryConditions11$[i]=$Eq11
\\76. $\ \ \ \ \ \ \ \ \ \ \ \ $ end do
\\77. $\ \ \ \ $ end do
\\78. Find the matrix $\partial_1$ from BoundaryConditions1 to BoundaryConditions11
\\79. Evaluate the rank of $\partial_1$\\
Algorithm to Find the Rank of the Matrix $\partial_2$
\\1. Begin
\\2. Input the value of $n\geq 1$
\\3. BoundaryConditions12$=[\ ]$
\\4. $\ \ \ \ $ do $i$ from $1:+3:3*n-2$
\\5. $\ \ \ \ $ Eq12$=e[i,i+1]-e[i,i+2]+e[i+1,i+2]$
\\6. $\ \ \ \ $ BoundaryConditions12$[i]=$Eq12
\\7. $\ \ \ \ $ end do
\\8. BoundaryConditions13$=[\ ]$
\\9. $\ \ \ \ $ do $i$ from $1:+3:3*n-2$
\\10. $\ \ \ \ \ \ \ \ \ \ \ \ $ do $j$ from $3*n+1$ to $5*n$
\\11. $\ \ \ \ \ \ \ \ \ \ \ \ \ \ $ Eq13$=e[i,j]-e[i,5*n+1]+e[j,5*n+1]$
\\12. $\ \ \ \ \ \ \ \ \ \ \ \ $ BoundaryConditions13$[i]=$Eq13
\\13. $\ \ \ \ \ \ \ \ \ \ \ \ $ end do
\\14. $\ \ \ \ $ end do
\\15. BoundaryConditions14$=[\ ]$
\\16. $\ \ \ \ $ do $i$ from $3:+3:3*n$
\\17. $\ \ \ \ \ \ \ \ \ \ \ \ $ do $j$ from $3*n+1$ to $5*n$
\\18. $\ \ \ \ \ \ \ \ \ \ \ \ \ \ \ $ Eq14$=e[i,j]-e[i,5*n+1]+e[j,5*n+1]$
\\19. $\ \ \ \ \ \ \ \ \ \ \ \ $ BoundaryConditions14$[i]=$Eq14
\\20. $\ \ \ \ \ \ \ \ \ \ \ \ $ end do
\\21. $\ \ \ \ $ end do
\\22. BoundaryConditions15$=[\ ]$
\\23. $\ \ \ \ $ do $i$ from $1$ to $2*n-2$
\\24. $\ \ \ \ \ \ \ \ \ \ \ \ $ do $j$ from $i+1$ to $2*n-1$
\\25. $\ \ \ \ \ \ \ \ \ \ \ \ \ \ \ \ \ \ \ $ do $k$ from $j+1$ to $2*n$
\\26. $\ \ \ \ \ \ \ \ \ \ \ \ \ \ \ \ \ \ \ $ Eq15$=e[3*n+i,3*n+j]-
\\27. $\ \ \ \ \ \ \ \ \ \ \ \ \ \ \ \ \ \ \ \ \ $ e[3*n+i,3*n+k]+
\\28. $\ \ \ \ \ \ \ \ \ \ \ \ \ \ \ \ \ \ \ \ \ $ e[3*n+j,3*n+k]$
\\29. $\ \ \ \ \ \ \ \ \ \ \ \ \ \ \ \ \ \ \ $ BoundaryConditions15$[i]=$Eq15
\\30. $\ \ \ \ \ \ \ \ \ \ \ \ \ \ \ \ \ \ \ $ end do
\\31. $\ \ \ \ \ \ \ \ \ \ \ \ $ end do
\\32. $\ \ \ \ $ end do
\\33. BoundaryConditions15$=[\ ]$
\\34. $\ \ \ \ $ do $i$ from $1$ to $2*n-1$
\\35. $\ \ \ \ \ \ \ \ \ \ \ \ $ do $j$ from $i+1$ to $2*n$
\\36. $\ \ \ \ \ \ \ \ \ \ \ \ $ Eq15$=e[3*n+i,3*n+j]-e[3*n+i,5*n+1]+e[3*n+j,5*n+1]$
\\37. $\ \ \ \ \ \ \ \ \ \ \ \ $ BoundaryConditions15$[i]=$Eq15
\\38. $\ \ \ \ \ \ \ \ \ \ \ \ $ end do
\\39. $\ \ \ \ $ end do
\\40. BoundaryConditions17$=[\ ]$
\\41. BoundaryConditions18$=[\ ]$
\\42. $\ \ \ \ $ do $i$ from $1$ to $3*n-2$
\\43. $\ \ \ \ \ \ \ \ \ \ \ \ $ do $k$ from 1:2:2*n-1
\\44. $\ \ \ \ \ \ \ \ \ \ \ \ $ Eq17$=e[i,i+1]-e[i,3*n+k]+e[i+1,3*n+k]$
\\45. $\ \ \ \ \ \ \ \ \ \ \ \ $ BoundaryConditions17$[i]=$Eq17
\\46. $\ \ \ \ \ \ \ \ \ \ \ \ $ Eq18$=e[i,i+1]-e[i,3*n+k+1]+e[i+1,3*n+k+1]$
\\47. $\ \ \ \ \ \ \ \ \ \ \ \ $ BoundaryConditions18$[i]=$Eq18
\\48. $\ \ \ \ \ \ \ \ \ \ \ \ $ i=i+3
\\49. $\ \ \ \ \ \ \ \ \ \ \ \ $ end do
\\50. $\ \ \ \ $ break
\\51. $\ \ \ \ $ end do
\\52. BoundaryConditions15$=[\ ]$
\\53. BoundaryConditions16$=[\ ]$
\\54. $\ \ \ \ $ do $i$ from $2$ to $3*n-1$
\\55. $\ \ \ \ \ \ \ \ \ \ \ \ $ do $k$ from 1:2:2*n
\\56. $\ \ \ \ \ \ \ \ \ \ \ \ $ Eq15$=e[i,i+1]-e[i,3*n+k]+e[i+1,3*n+k]$
\\57. $\ \ \ \ \ \ \ \ \ \ \ \ $ BoundaryConditions15$[i]=$Eq15
\\58. $\ \ \ \ \ \ \ \ \ \ \ \ $ Eq16$=e[i,i+1]-e[i,3*n+k+1]+e[i+1,3*n+k+1]$
\\59. $\ \ \ \ \ \ \ \ \ \ \ \ $ BoundaryConditions16$[i]=$Eq16
\\60. $\ \ \ \ \ \ \ \ \ \ \ \ $ i=i+3
\\61. $\ \ \ \ \ \ \ \ \ \ \ \ $ end do
\\62. $\ \ \ \ $ break
\\63. $\ \ \ \ $ end do
\\64. BoundaryConditions17$=[\ ]$
\\65. $\ \ \ \ $ do $i$ from $1:3:3*n-2$
\\66. $\ \ \ \ \ \ $ Eq17$=e[i,i+1]-e[i,5*n+1]+e[i+1,5*n+1]$
\\67. $\ \ \ \ $ BoundaryConditions17$[i]=$Eq17
\\68. $\ \ \ \ $ i=i+3
\\69. $\ \ \ \ $ end do
\\70. BoundaryConditions18$=[\ ]$
\\71. $\ \ \ \ $ do $i$ from $2:3:3*n-1$
\\72. $\ \ \ \ \ \ $ Eq18$=e[i,i+1]-e[i,5*n+1]+e[i+1,5*n+1]$
\\73. $\ \ \ \ $ BoundaryConditions18$[i]=$Eq18
\\74. $\ \ \ \ $ end do
\\75. BoundaryConditions19$=[\ ]$
\\76. BoundaryConditions20$=[\ ]$
\\77. $\ \ \ \ $ do $i$ from $1:3:3*n-5$
\\78. $\ \ \ \ \ \ \ \ \ \ \ \ $ do $j$ from $i+3:3:3*n-2$
\\79. $\ \ \ \ \ \ \ \ \ \ \ \ $ Eq19$=e[j-1,5*n+1]-e[i,5*n+1]+e[i,j-1]$
\\80. $\ \ \ \ \ \ \ \ \ \ \ \ $ BoundaryConditions19$[i]=$Eq19
\\81. $\ \ \ \ \ \ \ \ \ \ \ \ \ \ \ $ Eq20$=e[j,5*n+1]-e[i,5*n+1]+e[i,j]$
\\82. $\ \ \ \ \ \ \ \ \ \ \ \ $ BoundaryConditions20$[i]=$Eq20
\\83. $\ \ \ \ \ \ \ \ \ \ \ \ $ end do
\\84. $\ \ \ \ $ end do
\\85. BoundaryConditions21$=[\ ]$
\\86. BoundaryConditions22$=[\ ]$
\\87. $\ \ \ \ $ do $i$ from $3:3:3*n-3$
\\88. $\ \ \ \ \ \ \ \ \ \ \ \ $ do $j$ from $i+1:3:3*n-2$
\\89. $\ \ \ \ \ \ \ \ \ \ \ \ \ \ \ \ \ \ \ $ if i\small{$\sim$}=j
\\90. $\ \ \ \ \ \ \ \ \ \ \ \ \ \ \ \ \ \ \ \ \ \ \ \ \ \ \ \ \ $ Eq21$=e[j,5*n+1]-e[i,5*n+1]+e[i,j]$
\\91. $\ \ \ \ \ \ \ \ \ \ \ \ \ \ \ \ \ \ \ \ \ \ $ BoundaryConditions21$[i]=$Eq21
\\92. $\ \ \ \ \ \ \ \ \ \ \ \ \ \ \ \ \ \ \ \ \ \ $ Eq22$=e[j+2,5*n+1]-e[i,5*n+1]+e[i,j+2]$
\\93. $\ \ \ \ \ \ \ \ \ \ \ \ \ \ \ \ \ \ \ \ \ $ BoundaryConditions22$[i]=$Eq22
\\94. $\ \ \ \ \ \ \ \ \ \ \ \ \ \ \ \ \ \ \ \ \ $ end if
\\95. $\ \ \ \ \ \ \ \ \ \ \ \ \ \ \ $ end do
\\96. $\ \ \ \ \ \ $ end do
\\97. BoundaryConditions23$=[\ ]$
\\98. $\ \ \ \ $ do $i$ from $1:3:3*n-2$
\\99. $\ \ \ \ \ $ Eq23$=e[i,3*n]-e[i,5*n+1]+e[3*n,5*n+1]$
\\100. $\ \ \ \ \ $ BoundaryConditions23$[i]=$Eq23
\\101. $\ \ \ \ \ $ end do
\\102. BoundaryConditions24$=[\ ]$
\\103. BoundaryConditions25$=[\ ]$
\\104. $\ \ \ \ $ do $i$ from $2$ to $3*n-1$
\\105. $\ \ \ \ \ \ \ \ \ \ \ \ $ do $j$ from $1:2:2*n-1$
\\106. $\ \ \ \ \ \ \ \ \ \ \ \ $ Eq24$=e[i,3*n+j]-e[i,5*n+1]+e[3*n+j,5*n+1]$
\\107. $\ \ \ \ \ \ \ \ \ \ \ \ $ BoundaryConditions24$[i]=$Eq24
\\108. $\ \ \ \ \ \ \ \ \ \ \ \ $ Eq25$=e[i,3*n+j+1]-e[i,5*n+1]+e[3*n+j+1,5*n+1]$
\\109. $\ \ \ \ \ \ \ \ \ \ \ \ $ BoundaryConditions25$[i]=$Eq25
\\110. $\ \ \ \ \ \ \ \ \ \ \ \ $ i=i+3
\\111. $\ \ \ \ \ \ \ \ \ \ \ \ $ end do
\\112. $\ \ \ \ $ break
\\113. $\ \ \ \ $ end do
\\114. BoundaryConditions26$=[\ ]$
\\115. $\ \ \ \ $ do $i$ from $2$ to $3*n-1$
\\116. $\ \ \ \ \ \ \ \ \ \ \ \ $ do $j$ from $1:2:2*n-1$
\\117. $\ \ \ \ \ \ \ \ \ \ \ \ \ \ \ \ \ \ \ $ do $k$ from $j+1$ to $2*n$
\\118. $\ \ \ \ \ \ \ \ \ \ \ \ \ \ \ \ \ \ \ $ Eq26$=e[i,3*n+j]-e[i,3*n+k]+
\\119. $\ \ \ \ \ \ \ \ \ \ \ \ \ \ \ \ \ \ \ \ \ $ e[3*n+j,3*n+k]$
\\120. $\ \ \ \ \ \ \ \ \ \ \ \ \ \ \ \ \ \ \ $ BoundaryConditions26$[i]=$Eq26
\\121. $\ \ \ \ \ \ \ \ \ \ \ \ \ \ \ \ \ \ \ $ end do
\\122. $\ \ \ \ \ \ \ \ \ \ \ \ $ i=i+3
\\123. $\ \ \ \ \ \ \ \ \ \ \ \ $ end do
\\124. $\ \ \ \ $ break
\\125. $\ \ \ \ $ end do
\\126. BoundaryConditions27$=[\ ]$
\\127. $\ \ \ \ $ do $i$ from $2$ to $3*n-1$
\\128. $\ \ \ \ \ \ \ \ \ \ \ \ $ do $j$ from $2:2:2*n$
\\129. $\ \ \ \ \ \ \ \ \ \ \ \ \ \ \ \ \ \ \ $ do $k$ from $j+1$ to $2*n$
\\130. $\ \ \ \ \ \ \ \ \ \ \ \ \ \ \ \ \ \ \ $ Eq27$=e[i,3*n+j]-e[i,3*n+k]+
\\131. $\ \ \ \ \ \ \ \ \ \ \ \ \ \ \ \ \ \ \ \ \ $ e[3*n+j,3*n+k]$
\\132. $\ \ \ \ \ \ \ \ \ \ \ \ \ \ \ \ \ \ \ $ BoundaryConditions27$[i]=$Eq27
\\133. $\ \ \ \ \ \ \ \ \ \ \ \ \ \ \ \ \ \ \ $ end do
\\134. $\ \ \ \ \ \ \ \ \ \ \ \ $ i=i+3
\\135. $\ \ \ \ \ \ \ \ \ \ \ \ $ end do
\\136. $\ \ \ \ $ break
\\137. $\ \ \ \ $ end do
\\138. BoundaryConditions28$=[\ ]$
\\139. $\ \ \ \ $ do $i$ from $1$ to $3*n-5$
\\140. $\ \ \ \ \ \ \ \ \ \ \ \ $ do $j$ from $0:2:2*n-4$
\\141. $\ \ \ \ \ \ \ \ \ \ \ \ \ \ \ \ \ \ \ $ do $k$ from $j+2$ to $2*n-1$
\\142. $\ \ \ \ \ \ \ \ \ \ \ \ \ \ \ \ \ \ \ $ Eq28$=e[3*n-i,5*n-k]-
\\143. $\ \ \ \ \ \ \ \ \ \ \ \ \ \ \ \ \ \ \ \ \ \ \ \ $e[3*n-i,5*n-j]+e[5*n-k,5*n-j]$
\\144. $\ \ \ \ \ \ \ \ \ \ \ \ \ \ \ \ \ \ \ $ BoundaryConditions28$[i]=$Eq28
\\145. $\ \ \ \ \ \ \ \ \ \ \ \ \ \ \ \ \ \ \ $ end do
\\146. $\ \ \ \ \ \ \ \ \ \ \ \ $ i=i+3
\\147. $\ \ \ \ \ \ \ \ \ \ \ \ $ end do
\\148. $\ \ \ \ $ break
\\149. $\ \ \ \ $ end do
\\150. BoundaryConditions29$=[\ ]$
\\151. $\ \ \ \ $ do $i$ from $1$ to $3*n-5$
\\152. $\ \ \ \ \ \ \ \ \ \ \ \ $ do $j$ from $1:2:2*n-3$
\\153. $\ \ \ \ \ \ \ \ \ \ \ \ \ \ \ \ \ \ \ $ do $k$ from $j+1$ to $2*n-1$
\\154. $\ \ \ \ \ \ \ \ \ \ \ \ \ \ \ \ \ \ \ $ Eq29$=e[3*n-i,5*n-k]-
\\155. $\ \ \ \ \ \ \ \ \ \ \ \ \ \ \ \ \ \ \ \ \ \ \ \ $ e[3*n-i,5*n-j]+e[5*n-k,5*n-j]$
\\156. $\ \ \ \ \ \ \ \ \ \ \ \ \ \ \ \ \ \ \ $ BoundaryConditions29$[i]=$Eq29
\\157. $\ \ \ \ \ \ \ \ \ \ \ \ \ \ \ \ \ \ \ $ end do
\\158. $\ \ \ \ \ \ \ \ \ \ \ \ $ i=i+3
\\159. $\ \ \ \ \ \ \ \ \ \ \ \ $ end do
\\160. $\ \ \ \ $ break
\\161. $\ \ \ \ $ end do
\\162. BoundaryConditions30$=[\ ]$
\\163. BoundaryConditions31$=[\ ]$
\\164. $\ \ \ \ $ do $i$ from $1$ to $3*n-2$
\\165. $\ \ \ \ \ \ \ \ \ \ \ \ $ do $j$ from $1:2:2*n-1$
\\166. $\ \ \ \ \ \ \ \ \ \ \ \ $ Eq30$=e[i,i+2]-e[i,3*n+j]+e[i+2,3*n+j]$
\\167. $\ \ \ \ \ \ \ \ \ \ \ \ $ BoundaryConditions30$[i]=$Eq30
\\168. $\ \ \ \ \ \ \ \ \ \ \ \ $ Eq31$=e[i,i+2]-e[i,3*n+j+1]+e[i+2,3*n+j+1]$
\\169. $\ \ \ \ \ \ \ \ \ \ \ \ $ BoundaryConditions31$[i]=$Eq31
\\170. $\ \ \ \ \ \ \ \ \ \ \ \ $ i=i+3
\\171. $\ \ \ \ \ \ \ \ \ \ \ \ $ end do
\\172. $\ \ \ \ $ break
\\173. $\ \ \ \ $ end do
\\174. BoundaryConditions32$=[\ ]$
\\175. $\ \ \ \ $ do $i$ from $1$ to $3*n-2$
\\176. $\ \ \ \ \ \ \ \ \ \ \ \ $ do $j$ from $1:2:2*n-1$
\\177. $\ \ \ \ \ \ \ \ \ \ \ \ \ \ \ \ \ \ \ $ do $k$ from $j+1$ to $2*n$
\\178. $\ \ \ \ \ \ \ \ \ \ \ \ \ \ \ \ \ \ \ $ Eq32$=e[i,3*n+j]-e[i,3*n+k]+e[3*n+j,3*n+k]$
\\179. $\ \ \ \ \ \ \ \ \ \ \ \ \ \ \ \ \ \ \ $ BoundaryConditions32$[i]=$Eq32
\\180. $\ \ \ \ \ \ \ \ \ \ \ \ \ \ \ \ \ \ \ $ end do
\\181. $\ \ \ \ \ \ \ \ \ \ \ \ $ i=i+3
\\182. $\ \ \ \ \ \ \ \ \ \ \ \ $ end do
\\183. $\ \ \ \ $ break
\\184. $\ \ \ \ $ end do
\\185. BoundaryConditions33$=[\ ]$
\\186. $\ \ \ \ $ do $i$ from $3$ to $3*n-3$
\\187. $\ \ \ \ \ \ \ \ \ \ \ \ $ do $j$ from $2:2:2*n-2$
\\188. $\ \ \ \ \ \ \ \ \ \ \ \ \ \ \ \ \ \ \ $ do $k$ from $j+1$ to $2*n$
\\189. $\ \ \ \ \ \ \ \ \ \ \ \ \ \ \ \ \ \ \ $ Eq33$=e[i,3*n+j]-e[i,3*n+k]+e[3*n+j,3*n+k]$
\\190. $\ \ \ \ \ \ \ \ \ \ \ \ \ \ \ \ \ \ \ $ BoundaryConditions33$[i]=$Eq33
\\191. $\ \ \ \ \ \ \ \ \ \ \ \ \ \ \ \ \ \ \ $ end do
\\192. $\ \ \ \ \ \ \ \ \ \ \ \ $ i=i+3
\\193. $\ \ \ \ \ \ \ \ \ \ \ \ $ end do
\\194. $\ \ \ \ $ break
\\195. $\ \ \ \ $ end do
\\196. BoundaryConditions34$=[\ ]$
\\197. $\ \ \ \ $ do $i$ from $2$ to $3*n-4$
\\198. $\ \ \ \ \ \ \ \ \ \ \ \ $ do $j$ from $1:2:2*n-3$
\\199. $\ \ \ \ \ \ \ \ \ \ \ \ \ \ \ \ \ \ \ $ do $k$ from $j+1$ to $2*n-1$
\\200. $\ \ \ \ \ \ \ \ \ \ \ \ \ \ \ \ \ \ \ $ Eq34$=e[3*n-i,5*n-k]-e[3*n-i,5*n-j]+e[5*n-k,5*n-j]$
\\201. $\ \ \ \ \ \ \ \ \ \ \ \ \ \ \ \ \ \ \ $ BoundaryConditions34$[i]=$Eq34
\\202. $\ \ \ \ \ \ \ \ \ \ \ \ \ \ \ \ \ \ \ $ end do
\\203. $\ \ \ \ \ \ \ \ \ \ \ \ $ i=i+3
\\204. $\ \ \ \ \ \ \ \ \ \ \ \ $ end do
\\205. $\ \ \ \ $ break
\\206. $\ \ \ \ $ end do
\\207. BoundaryConditions35$=[\ ]$
\\208. $\ \ \ \ $ do $i$ from $0$ to $3*n-3$
\\209. $\ \ \ \ \ \ \ \ \ \ \ \ $ do $j$ from $0:2:2*n-2$
\\210. $\ \ \ \ \ \ \ \ \ \ \ \ \ \ \ \ \ \ \ $ do $k$ from $j+1$ to $2*n-1$
\\211. $\ \ \ \ \ \ \ \ \ \ \ \ \ \ \ \ \ \ \ $ Eq35$=e[3*n-i,5*n-k]-e[3*n-i,5*n-j]+e[5*n-k,5*n-j]$
\\212. $\ \ \ \ \ \ \ \ \ \ \ \ \ \ \ \ \ \ \ $ BoundaryConditions35$[i]=$Eq35
\\213. $\ \ \ \ \ \ \ \ \ \ \ \ \ \ \ \ \ \ \ $ end do
\\214. $\ \ \ \ \ \ \ \ \ \ \ \ $ i=i+3
\\215. $\ \ \ \ \ \ \ \ \ \ \ \ $ end do
\\216. $\ \ \ \ $ break
\\217. $\ \ \ \ $ end do
\\218. Find the matrix $\partial_2$ from BoundaryConditions8 to BoundaryConditions35
\\219. Evaluate the rank of $\partial_2$

We can use MATLAB \cite{MAL} to see that the ranks of boundary homomorphisms $\partial_i$ with $i=1, 2$ are given as $rank(Im\ \partial_1)=5n$ and $rank(Im\ \partial_2)=10 n^2$. By dimension theorem of vector spaces, the nullities of $\partial_i$ with $i=1, 2$ can be expressed as $rank(Ker\ \partial_1)=\alpha_1-rank (Im\ \partial_1)=10 n^2$ and $rank(Ker\ \partial_2)=\alpha_2-rank(Im\ \partial_2)=\alpha_2-10 n^2$. Thus, the Betti numbers of $\Delta_{T}(F_{5n+1})$ are given by $(\beta_0,\beta_1,\beta_2,\beta_3,\ldots)=(1, 0, \frac{4n^{3}+12n^{2}+14n}{3}, 0, \ldots)$, see Equation (\ref{eq2}). Hence proved.
\endproof
\begin{Theorem}\label{t6}
The TSC associated to a family of friendship graphs $F_{5n+1}$ is CM.
\end{Theorem}
\proof
Since the first homology group of $\Delta_{T}(F_{5n+1})$ is trivial, see Lemma \ref{t5}. Therefore, $\Delta_{T}(F_{5n+1})$ is CM by Theorem \ref{t3}.
\endproof
\begin{Remark}
The TSC associated to a family of friendship graphs $F_{5n+1}$ is a bouquet of $(\frac{4n^{3}+12n^{2}+14n}{3})$ two dimensional spheres with $n\geq 1$.
\end{Remark}

\begin{Corollary}\label{c1}
The TSC associated to a family of friendship graphs $F_{5n+1}$ is unmixed.
\end{Corollary}
\begin{proof}
We know that every CM simplicial complex is unmixed \cite{F1}. By Theorem \ref{t6}, $\Delta_{T}(F_{5n+1})$ is unmixed.   
\end{proof}

We compute the associated primes of the facet ideal $I_{\mathcal{F}}(\Delta_T(F_{5n+1}))$. 

\begin{Theorem}\label{t7}
	Let  $\Delta_{T}({F}_{5n+1})$ be TSC of a family of friendship graphs $F_{5n+1}$. Then the associated primes of  the facet ideal $I_\mathcal{F}(\Delta_{T}(F_{5n+1}))$  are given by\\
	$$I_\mathcal{F}(\Delta_{T}(F_{5n+1}))=\bigcap\limits_{r=1}^{10}( x_{1},\ldots,\hat x_{j_{r}},\ldots,x_{5n+1})$$
	for all $j_1\in\{3n+1,i\acute{}_1,i_2,i\acute{}_2,\ldots,i_n,i\acute{}_n\}$ and $i\acute{}_1\in\{1,2,3n+2\}$ with $i_k<i\acute{}_k\in\{3k-2,3k-1,3k\}$ such that $k=2,\ldots,n$;
	$j_2\in\{i_{1},3n+2,i_2,i\acute{}_2,\ldots,i_n,i\acute{}_n\}$  and  $i_{1}\in\{2,3\}$ with $i_k<i\acute{}_k\in\{3k-2,3k-1,3k\}$ such that $k=2,\ldots,n$;
	$j_3\in\{ i_{1},i\acute{}_1,\ldots, i_j,i\acute{}_j,\ldots,i_n,i\acute{}_n\}$ with corresponding values of $i_{1}=3n+1,1$ and $i\acute{}_1=3,3n+2$; corresponding values of $i_{j}=3n+2j-1,3j-2$ and $i\acute{}_j=3j,3n+2j$ such that $i_k<i\acute{}_k\in\{3k-2,3k-1,3k\}$, where $j,k=2,\ldots,n$ for $j\neq k$;
	$j_4\in\{i_{1},i\acute{}_1,\ldots,i_n,i\acute{}_n\}$ with corresponding values of $i_{1}=3n+1,1$ and $i\acute{}_1=3,3n+2$
	such that $i_k<i\acute{}_k\in\{3k-2,3k-1,3k\}$, where  $k=2,\ldots,n;$
	$j_5\in\{i_{1},i\acute{}_1,\ldots, i_j,i\acute{}_j,\ldots,i_n,i\acute{}_n\}$ with $i_{1}<i\acute{}_1\in\{1,2,3\}$, $i_j=3n+2j-1$, $i\acute{}_j\in\{3j-2,3j-1,3j,3n+2j\}$ and $i_k<i\acute{}_k\in\{3k-2,3k-1,3k\}$, where $j,k=2,\ldots,n$ such that $j\neq k$;
	$j_6\in\{i_{1},i\acute{}_1,\ldots,i_n,i\acute{}_n\}$ with $i_k<i\acute{}_k\in\{3k-2,3k-1,3k\}$ such that $k=1,\ldots,n$;
	$j_7\in\{ i_{1},i\acute{}_1,\ldots,i_j,i\acute{}_j,\ldots,i_l,i\acute{}_l,\ldots,i_n,i\acute{}_n\}$ with $i_{1}<i\acute{}_1\in\{1,2,3\}$, $i_j=3n+2j-1,i\acute{}_j=3j,i_l=3n+2l-1,i\acute{}_l=3l$ and  $i_k<i\acute{}_k\in\{3k-2,3k-1,3k\}$, where $j, k, l=2,\ldots,n$ such that $j<l$ and $j\neq k\neq l$;
	$j_8\in\{i_{1},i\acute{}_1,\ldots,i_j,i\acute{}_j,\ldots,i_l,i\acute{}_l,\ldots,i_n,i\acute{}_n\}$ with $i_{1}<i\acute{}_1\in\{1,2,3\}$, $i_j=3n+2j-1$, $i\acute{}_j=3j$, $i_l=3l-2$, $i\acute{}_l=3n+2l$ and $i_k<i\acute{}_k\in\{3k-2,3k-1,3k\}$, where $j, k, l=2,\ldots,n$ such that  $j\neq k\neq l$;
	$j_9\in\{i_{1},i\acute{}_1,\ldots,i_j,i\acute{}_j,\ldots,i_l,i\acute{}_l,\ldots,i_n,i\acute{}_n\}$ with $i_{1}<i\acute{}_1\in\{1,2,3\}$, $i_j=3j-2$, $i\acute{}_j=3n+2j$, $i_l=3l-2$, $i\acute{}_l=3n+2l$ and $i_k<i\acute{}_k\in\{3k-2,3k-1,3k\}$, where $j, k, l=2,\ldots,n$ such that  $j\neq k\neq l$ and $j<l$;
	$j_{10}\in\{i_{1},i\acute{}_1,\ldots,i_l,i\acute{}_l,\ldots,i_n,i\acute{}_n\}$ with $i_{1}<i\acute{}_1\in\{1,2,3\}$, $i_l\in\{3l-2,3l-1,3l\}$, $i\acute{}_l=3n+2l$ and  $i_k<i\acute{}_k\in\{3k-2,3k-1,3k\}$, where $l\in\{2,\ldots,n\}$ and $k=2,\ldots,n$ such that  $k\neq l$.
\end{Theorem}
\begin{proof} 
	A minimal prime ideal of the facet ideal $I_{\mathcal{F}}(\Delta)$ has a one-to-one correspondence with the minimal vertex cover of the simplicial complex, according to \cite{F1}, Proposition 1.8. Therefore, it is sufficient to compute all the minimal vertex covers of $\Delta_T(F_{5n+1})$ in order to compute the primary decomposition of the facet ideal $I_{\mathcal{F}}(\Delta_T(F_{5n+1}))$. We find all possible minimal vertex covers of $\Delta_{T}({F}_{5n+1})$ into $r$ steps with $r=1,\ldots,10$. The minimal vertex covers of $\Delta_{T}({F}_{5n+1})$ in first step can be expressed as $( x_{1},\ldots,\hat x_{j_1},\ldots,x_{5n+1})$ for all
	$j_1\in\{3n+1,1,4,5,\ldots,3n-5,3n-4,3n-2,3n-1\}$; 
	$j_1\in\{3n+1,1,4,5,\ldots,3n-5,3n-4,3n-2,3n\}$;
	$j_1\in\{3n+1,1,4,5,\ldots,3n-5,3n-4,3n-1,3n\}$;
	$j_1\in\{3n+1,1,4,5,\ldots,3n-5,3n-3,3n-2,3n-1\}$; 
	$j_1\in\{3n+1,1,4,5,\ldots,3n-5,3n-3,3n-2,3n\}$;
	$j_1\in\{3n+1,1,4,5,\ldots,3n-5,3n-3,3n-1,3n\}$; 
	$j_1\in\{3n+1,1,4,5,\ldots,3n-4,3n-3,3n-2,3n-1\}$; 
	$j_1\in\{3n+1,1,4,5,\ldots,3n-4,3n-3,3n-2,3n\}$; 
	$j_1\in\{3n+1,1,4,5,\ldots,3n-4,3n-3,3n-1,3n\}$;\ldots; 
	$j_1\in\{3n+1,1,4,6,\ldots,3n-5,3n-4,3n-2,3n-1\}$; 
	$j_1\in\{3n+1,1,4,6,\ldots,3n-5,3n-4,3n-2,3n\}$; 
	$j_1\in\{3n+1,1,4,6,\ldots,3n-5,3n-4,3n-1,3n\}$; 
	$j_1\in\{3n+1,1,4,6,\ldots,3n-5,3n-3,3n-2,3n-1\}$; 
	$j_1\in\{3n+1,1,4,6,\ldots,3n-5,3n-3,3n-2,3n\}$; 
	$j_1\in\{3n+1,1,4,6,\ldots,3n-5,3n-3,3n-1,3n\}$; 
	$j_1\in\{3n+1,1,4,6,\ldots,3n-4,3n-3,3n-2,3n-1\}$; 
	$j_1\in\{3n+1,1,4,6,\ldots,3n-4,3n-3,3n-2,3n\}$; 
	$j_1\in\{3n+1,1,4,6,\ldots,3n-4,3n-3,3n-1,3n\}$; 
	$j_1\in\{3n+1,1,5,6,\ldots,3n-5,3n-4,3n-2,3n-1\}$; 
	$j_1\in\{3n+1,1,5,6,\ldots,3n-5,3n-4,3n-2,3n\}$; 
	$j_1\in\{3n+1,1,5,6,\ldots,3n-5,3n-4,3n-1,3n\}$; 
	$j_1\in\{3n+1,1,5,6,\ldots,3n-5,3n-3,3n-2,3n-1\}$; 
	$j_1\in\{3n+1,1,5,6,\ldots,3n-5,3n-3,3n-2,3n\}$; 
	$j_1\in\{3n+1,1,5,6,\ldots,3n-5,3n-3,3n-1,3n\}$; 
	$j_1\in\{3n+1,1,5,6,\ldots,3n-4,3n-3,3n-2,3n-1\}$; 
	$j_1\in\{3n+1,1,5,6,\ldots,3n-4,3n-3,3n-2,3n\}$; 
	$j_1\in\{3n+1,1,5,6,\ldots,3n-4,3n-3,3n-1,3n\}$; 
	$j_1\in\{3n+1,2,4,5,\ldots,3n-5,3n-4,3n-2,3n-1\}$; 
	$j_1\in\{3n+1,2,4,5,\ldots,3n-5,3n-4,3n-2,3n\}$; 
	$j_1\in\{3n+1,2,4,5,\ldots,3n-5,3n-4,3n-1,3n\}$; 
	$j_1\in\{3n+1,2,4,5,\ldots,3n-5,3n-3,3n-2,3n-1\}$; 
	$j_1\in\{3n+1,2,4,5,\ldots,3n-5,3n-3,3n-2,3n\}$; 
	$j_1\in\{3n+1,2,4,5,\ldots,3n-5,3n-3,3n-1,3n\}$; 
	$j_1\in\{3n+1,2,4,5,\ldots,3n-4,3n-3,3n-2,3n-1\}$; 
	$j_1\in\{3n+1,2,4,5,\ldots,3n-4,3n-3,3n-2,3n\}$; 
	$j_1\in\{3n+1,2,4,5,\ldots,3n-4,3n-3,3n-1,3n\}$; 
	$j_1\in\{3n+1,2,4,6,\ldots,3n-5,3n-4,3n-2,3n-1\}$; 
	$j_1\in\{3n+1,2,4,6,\ldots,3n-5,3n-4,3n-2,3n\}$; 
	$j_1\in\{3n+1,2,4,6,\ldots,3n-5,3n-4,3n-1,3n\}$; 
	$j_1\in\{3n+1,2,4,6,\ldots,3n-5,3n-3,3n-2,3n-1\}$; 
	$j_1\in\{3n+1,2,4,6,\ldots,3n-5,3n-3,3n-2,3n\}$; 
	$j_1\in\{3n+1,2,4,6,\ldots,3n-5,3n-3,3n-1,3n\}$; 
	$j_1\in\{3n+1,2,4,6,\ldots,3n-4,3n-3,3n-2,3n-1\}$; 
	$j_1\in\{3n+1,2,4,6,\ldots,3n-4,3n-3,3n-2,3n\}$; 
	$j_1\in\{3n+1,2,4,6,\ldots,3n-4,3n-3,3n-1,3n\}$; 
	$j_1\in\{3n+1,2,5,6,\ldots,3n-5,3n-4,3n-2,3n-1\}$; 
	$j_1\in\{3n+1,2,5,6,\ldots,3n-5,3n-4,3n-2,3n\}$; 
	$j_1\in\{3n+1,2,5,6,\ldots,3n-5,3n-4,3n-1,3n\}$; 
	$j_1\in\{3n+1,2,5,6,\ldots,3n-5,3n-3,3n-2,3n-1\}$; 
	$j_1\in\{3n+1,2,5,6,\ldots,3n-5,3n-3,3n-2,3n\}$; 
	$j_1\in\{3n+1,2,5,6,\ldots,3n-5,3n-3,3n-1,3n\}$; 
	$j_1\in\{3n+1,2,5,6,\ldots,3n-4,3n-3,3n-2,3n-1\}$; 
	$j_1\in\{3n+1,2,5,6,\ldots,3n-4,3n-3,3n-2,3n\}$; 
	$j_1\in\{3n+1,2,5,6,\ldots,3n-4,3n-3,3n-1,3n\}$; 
	$j_1\in\{3n+1,3n+2,4,5,\ldots,3n-5,3n-4,3n-2,3n-1\}$; 
	$j_1\in\{3n+1,3n+2,4,5,\ldots,3n-5,3n-4,3n-2,3n\}$; 
	$j_1\in\{3n+1,3n+2,4,5,\ldots,3n-5,3n-4,3n-1,3n\}$; 
	$j_1\in\{3n+1,3n+2,4,5,\ldots,3n-5,3n-3,3n-2,3n-1\}$; 
	$j_1\in\{3n+1,3n+2,4,5,\ldots,3n-5,3n-3,3n-2,3n\}$; 
	$j_1\in\{3n+1,3n+2,4,5,\ldots,3n-5,3n-3,3n-1,3n\}$; 
	$j_1\in\{3n+1,3n+2,4,5,\ldots,3n-4,3n-3,3n-2,3n-1\}$; 
	$j_1\in\{3n+1,3n+2,4,5,\ldots,3n-4,3n-3,3n-2,3n\}$; 
	$j_1\in\{3n+1,3n+2,4,5,\ldots,3n-4,3n-3,3n-1,3n\}$; 
	$j_1\in\{3n+1,3n+2,4,6,\ldots,3n-5,3n-4,3n-2,3n-1\}$; 
	$j_1\in\{3n+1,3n+2,4,6,\ldots,3n-5,3n-4,3n-2,3n\}$; 
	$j_1\in\{3n+1,3n+2,4,6,\ldots,3n-5,3n-4,3n-1,3n\}$; 
	$j_1\in\{3n+1,3n+2,4,6,\ldots,3n-5,3n-3,3n-2,3n-1\}$; 
	$j_1\in\{3n+1,3n+2,4,6,\ldots,3n-5,3n-3,3n-2,3n\}$; 
	$j_1\in\{3n+1,3n+2,4,6,\ldots,3n-5,3n-3,3n-1,3n\}$; 
	$j_1\in\{3n+1,3n+2,4,6,\ldots,3n-4,3n-3,3n-2,3n-1\}$; 
	$j_1\in\{3n+1,3n+2,4,6,\ldots,3n-4,3n-3,3n-2,3n\}$; 
	$j_1\in\{3n+1,3n+2,4,6,\ldots,3n-4,3n-3,3n-1,3n\}$; 
	$j_1\in\{3n+1,3n+2,5,6,\ldots,3n-5,3n-4,3n-2,3n-1\}$; 
	$j_1\in\{3n+1,3n+2,5,6,\ldots,3n-5,3n-4,3n-2,3n\}$; 
	$j_1\in\{3n+1,3n+2,5,6,\ldots,3n-5,3n-4,3n-1,3n\}$; 
	$j_1\in\{3n+1,3n+2,5,6,\ldots,3n-5,3n-3,3n-2,3n-1\}$; 
	$j_1\in\{3n+1,3n+2,5,6,\ldots,3n-5,3n-3,3n-2,3n\}$; 
	$j_1\in\{3n+1,3n+2,5,6,\ldots,3n-5,3n-3,3n-1,3n\}$; 
	$j_1\in\{3n+1,3n+2,5,6,\ldots,3n-4,3n-3,3n-2,3n-1\}$; 
	$j_1\in\{3n+1,3n+2,5,6,\ldots,3n-4,3n-3,3n-2,3n\}$; 
	$j_1\in\{3n+1,3n+2,5,6,\ldots,3n-4,3n-3,3n-1,3n\}$.
	Thus, the minimal vertex covers of $\Delta_{T}({F}_{5n+1})$ in Step 1 can be written as
	$( x_{1},\ldots,\hat x_{j_1},\ldots, x_{5n+1})$ for all  $j_1\in\{3n+1,i\acute{}_1,i_2,i\acute{}_2,\ldots,i_n,i\acute{}_n\}$ with $i\acute{}_1\in\{1,2,3n+2\}$ and $i_k<i\acute{}_k\in\{3k-2,3k-1,3k\}$ such that $k=2,\ldots,n$. One can establish the other steps $r=2,\ldots,10$ in a similar manner. Hence proved.
\end{proof}

\begin{Proposition}
	Let $\Delta_{T}({F}_{5n+1})$ be TSC associated to a family of friendship graphs $F_{5n+1}$. Then the cardinality of each minimal vertex cover of $\Delta_{T}({F}_{5n+1})$ is $3n+1$.The number of minimal vertex covers of $\Delta_{T}({F}_{5n+1})$ is $3^{n-2}(2n^{2}+19n+9)$.
\end{Proposition}
\begin{proof}
Note that the cardinality of each minimal vertex cover of $\Delta_{T}({F}_{5n+1})$ is $3n+1$, see Definition \ref{dd} and Corollary \ref{c1}. One can see from Theorem \ref{t7} that the total number of minimal vertex covers of $\Delta_{T}({F}_{5n+1})$ in $r$ steps with $r=1,\ldots,10$ is given by\\
	$3^{n}+2.3^{n-1}+4(n-1)3^{n-2}+2.3^{n-1}+12(n-1)3^{n-2}+3^n+(1+\ldots+n-2)3^{n-2}+2(1+\ldots+n-2)3^{n-2}+(1+\ldots+n-2)3^{n-2}+(n-1)3^{n}
	=3^{n}+2.3^{n-1}+4(n-1)3^{n-2}+2.3^{n-1}+12(n-1)3^{n-2}+3^n+\frac{1}{2}(n-1)(n-2)3^{n-2}+(n-1)(n-2)3^{n-2}+\frac{1}{2}(n-1)(n-2)3^{n-2}+(n-1)3^{n}
	=9.3^{n-2}+6.3^{n-2}+4(n-1)3^{n-2}+6.3^{n-2}+12(n-1)3^{n-2}+9.3^{n-2}+\frac{1}{2}(n-1)(n-2)3^{n-2}+(n-1)(n-2)3^{n-2}+\frac{1}{2}(n-1)(n-2)3^{n-2}+9(n-1)3^{n-2}
	=3^{n-2}(2n^{2}+19n+9)$. Hence the result.
\end{proof}
{\bf Acknowledgement:}
The Higher Education Commission of Pakistan is acknowledged by the authors for its partial support.
%
%


{\small
\begin{thebibliography}{999}
\bibitem{AM} Ahmed, I., Muhmood, S. (2020). Computational aspects of Line Simplicial Complexes. J. Intell. Fuzzy Syst. 39(1):35-42. DOI: 10.3233/JIFS-190369
\bibitem{B2} Behzad, M. (1970). A Characterization of Total Graphs. Proc. Amer. Math. Soc. 26:383-389. DOI: 10.2307/2037344
\bibitem{B3} Behzad, M. (1969). The Connectivity of Total Graphs. Bull. Aust. Math. Soc. 1(2):175-181. DOI: 10.1017/S0004972700041423
\bibitem{B1} Behzad, M. (1967). A Criterion for the Planarity of the Total Graph of a Graph. Proc. Cambridge
Philos. Soc. 63:679-681. DOI: 10.1017/S0305004100041657
\bibitem{BC} Behzad, M., Chartrand, G. (1966/67). Total graphs and Traversability. Proc. Edinb. Math. Soc. 15(2):117-120. DOI: 10.1017/S0013091500011421
\bibitem{BR2} Behzad, M., Radjavi, H. (1969). Structure of Regular Total graphs. J. Lond. Math. Soc. 44(1):433-436. DOI: 10.1112/jlms/s1-44.1.433
\bibitem{BR1} Behzad, M., Radjavi, H. (1968). The Total Group of a Graph. Proc. Amer. Math. Soc. 19:158-163. DOI: 10.2307/2036158
\bibitem{C} Chen, L., Rong, Y. (2010). Digital Topological Method for Genus and Betti Numbers. Topol. Appl. 157: 1931-1936. DOI: 10.1016/j.topol.2010.04.006
\bibitem{F1} Faridi, S. (2002). The facet ideal of a simplicial complex. Manuscripta Math. 109:159-174. DOI: 10.1007/s00229-002-0293-9
\bibitem{F} Faridi, S. (2005). Cohen-Macaulay Properties of Square-Free Monomial Ideals. J. Comb. Theory Ser. A. 109(2):299-329. DOI: 10.1016/j.jcta.2004.09.005
\bibitem{F1} Faridi, S. (2004). Simplicial Tree are Sequentially Cohen-Macaulay. J. Pure and Applied Algebra. 190: 121-136.
\bibitem{GSG} Garg, P., Sinha, D., Goyal, S. (2015). Eulerian and Hamiltonian Properties of Gallai and anti-Gallai Total Graphs. J. Indones. Math. Soc. 21(2):105-116. DOI: 10.22342/jims.21.2.230.105-116
\bibitem{Ha2} Haghighi, H., Yassemi, S., Zaare-Nahandi, R. (2012). A generalization of k-Cohen-Macaulay simplicial complexes. Ark. Mat. 50:279-290. DOI: 10.1007/s11512-010-0136-y
\bibitem{Ho} Hochster, M. (1975/77). Cohen-Macaulay rings, combinatorics, and simplicial complexes: Ring Theory II (Proc. Second Conf., Univ. Oklahoma, Norman, Okla.), 171-223. Lecture Notes in Pure and Appl. Math., Vol. 26, Dekker, New York. 
\bibitem{MABY} Mahmood, H., Anwar, I., Binyamin, M.A., Yasmeen, S. (2017).
On the Connectedness of $f$-Simplicial Complexes. J. Algebra its Appl. 16(1):1750017-1750026. DOI: 10.1142/S0219498817500177
\bibitem{MAL} Muhmood, S., Ahmed, I., Liaquat, A. (2019). Gallai Simplicial Complexes. J. Intell. Fuzzy Syst. 36(6):5645-5651. DOI: 10.3233/JIFS-181478
\bibitem{Re} Reisner, G.A. (1976). Cohen-Macaulay Quotients of Polynomial Rings. Advances in Math. 21(1):30-49. DOI: 10.1016/0001-8708(76)90114-6
\bibitem{R} Rotman, J.J. (1988). An Introduction to Algebraic Topology. Springer-Verlag, New York. DOI: 10.1007/978-1-4612-4576-6
\bibitem{Sc} Schenzel, P. (1981). On the Number of Faces of Simplicial Complexes and the Purity of Frobenius. Math. Z., 178:125-142. DOI: 10.1007/BF01218376 
\bibitem{S1} Stanley, R.P. (1995). Combinatorics and Commutative Algebra, Second Edition, Birkh\"{a}user, Boston. DOI: 10.1007/b139094
\bibitem{S} Stanley, R.P. (1977). Cohen-Macaulay Complexes. In: Aigner, M. (eds) Higher Combinatorics. NATO Advanced Study Inst. Ser. vol 31. Springer, Dordrecht. DOI: 10.1007/978-94-010-1220-1\_3
\end{thebibliography}}
\end{document}